\documentclass[11pt,a4paper]{article} 
\usepackage[no-comment]{WSKtracking}
\usepackage{ifpdf}

\newcommand{\Title}{Optimal Coupling of Jumpy Brownian Motion on the Circle}
\usepackage[british]{babel} 

\usepackage{calc, xspace} \usepackage{amsmath, amsthm, amssymb,
  amsfonts, euscript} \usepackage{enumerate}
\usepackage{bbm}

\usepackage[longnamesfirst]{natbib} \setcitestyle{round}

\ifpdf \usepackage{xcolor} \input dvipsnam.def \else
\usepackage{color} \fi \usepackage{graphicx}

\definecolor{linkcolor}{named}{Maroon}
\definecolor{citecolor}{named}{PineGreen}
\definecolor{urlcolor}{named}{RoyalPurple}
\definecolor{okcolor}{named}{OliveGreen}
\definecolor{alertcolor}{named}{BrickRed}

\ifpdf \DeclareGraphicsExtensions{.png,.jpg}%
\else \DeclareGraphicsExtensions{.eps,.ps,.png,.jpg}%
\fi

\usepackage[left=2.2cm,right=2.2cm,top=3cm,bottom=3cm]{geometry}
\parskip \medskipamount

\theoremstyle{plain} \newtheorem{thm}{Theorem}
 
\newtheorem{lem}[thm]{Lemma} 
\theoremstyle{definition} \newtheorem{defn}[thm]{Definition}
\theoremstyle{remark} \newtheorem{rem}[thm]{Remark}

\newcommand{\R}{\mathbb{R}}
\newcommand{\Ex}[2][]{\operatorname{\mathbb{E}}_{#1}\left[#2\right]}
\newcommand{\Indicator}[1]{\operatorname{\mathbbm{1}}_{\left[#1\right]}}
\newcommand{\prob}[1]{\operatorname{\mathbb{P}}\left(#1\right)}

\newcommand{\bra}[1]{\left(#1\right)}

 \newcommand{\Z}{\mathbb Z}
\newcommand{\set}[1]{\left\{#1\right\}}

\newcommand{\dd}{\ensuremath{\textrm{d}}}

\newcommand{\distname}[1]{\mbox{\textup{#1}}}

\newcommand{\expdist}{\distname{Exp}}
\newcommand{\berdist}{\distname{Bernoulli}}
\newcommand{\unifdist}{\distname{Uniform}}
\newcommand{\iid}{\ensuremath{\stackrel{\footnotesize{\textup{i.i.d.}}}{\sim}}}

\newcommand{\law}{\mathcal L}
\newcommand{\eqdist}{\ensuremath{\stackrel{\mathcal{D}}{=}}}
\DeclareMathOperator{\cosech}{cosech}
\DeclareMathOperator{\sech}{sech}

\newcommand{\JBM}{JBM($\lambda$)}
\newcommand{\Tc}{\ensuremath{T_\mathrm{couple}}}
\renewcommand{\Tc}{\ensuremath{T}}
\newcommand{\tv}[1]{{\Vert{#1}\Vert}_\mathrm{TV}}
\newcommand{\modtwopi}{\ (\mathrm{mod}\ 2\pi)}
\newcommand{\Tmin}{\ensuremath{T_{\mathrm{min}}}}
\newcommand{\Tm}{\ensuremath{T_{\mathrm{M}}}}
\newcommand{\Ts}{\ensuremath{T_{\mathrm{S}}}}
\newcommand{\sync}{{}_\mathrm{S}}
\newcommand{\mirror}{{}_\mathrm{M}}
\newcommand{\interval}{\ensuremath{(-\pi,\pi]}} 

\usepackage[ \ifpdf pdftex, pdfstartview=FitH, \fi ]{hyperref}
\hypersetup{pdfstartview=FitH} \hypersetup{citecolor=citecolor}
\hypersetup{linkcolor=linkcolor}
\hypersetup{urlcolor=urlcolor} \hypersetup{colorlinks=true}

\usepackage{framed} \usepackage[small]{caption}

{\begin{figure}[htbp]\begin{framed}}%
    {\end{framed}\end{figure}}
 \begin{document}

 \title{\Title}
 \author{%
Stephen B. Connor
\ and Roberta Merli}

 \date{1st May 2023}
 \maketitle

\TrackingPreamble

 \begin{abstract}
 	Consider a Brownian motion on the circumference of the unit circle, which jumps to the opposite point of the circumference at incident times of an independent Poisson process of rate $\lambda$. We examine the problem of coupling two copies of this `jumpy Brownian motion' started from different locations, so as to optimise certain functions of the coupling time. We describe two intuitive co-adapted couplings (`Mirror' and `Synchronous') which differ only when the two processes are directly opposite one another, and show that the question of which strategy is best depends upon the jump rate $\lambda$ in a non-trivial way.
 	
 	More precisely, we use the theory of stochastic control to show that there exists a critical value $\lambda^\star = 0.083\dots$ such that the Mirror coupling minimises the mean coupling time within the class of all co-adapted couplings when $\lambda<\lambda^\star$, but for $\lambda\ge \lambda^\star$ the Synchronous coupling uniquely maximises the Laplace transform $\Ex{e^{-\gamma T}}$ of all coupling times $T$ within this class. We also provide an explicit description of a (non co-adapted) maximal coupling for any jump rate in the case that the two jumpy Brownian motions begin at antipodal points of the circle.
 \end{abstract}	

 \begin{quotation}
  \noindent
  Keywords and phrases:\\
  \textsc{Optimal coupling; co-adapted coupling; mirror and synchronous coupling; stochastic control; HJB equation}
 \end{quotation}

 \begin{quotation}\noindent
   \end{quotation}



\section{Introduction}\label{sec:intro}


Consider a continuous time stochastic process $X=(X_t)_{t\ge 0}$ on \interval\ \xNB{Changed from $[0,2\pi)$: is this ok?} given by
\begin{equation}\label{eq:JBM_defn}
	X_t= X_0 + \frac{1}{2}B_t+\pi N^\lambda_t \modtwopi \,,
\end{equation}
where $B_t$ is a standard $\R$-valued Brownian motion and $N^\lambda$ is an independent Poisson process of rate $\lambda\ge0$. (The factor of a half is introduced purely for algebraic convenience in what follows.)
$X$ is a L\'evy process which can be viewed as a Brownian motion on the circumference of the unit circle which jumps to the antipodal point of the circle at incident times of $N^\lambda$. For this reason, we will refer to the process $X$ as a \textit{jumpy Brownian motion of rate $\lambda$} (\JBM).

In this paper we are interested in couplings of two \JBM\ processes started from different points of the circle. That is, we are interested in processes $(X,\hat X)$ on $\interval^2$ such that, viewed marginally, $X$ and $\hat X$ each behave as a copy of \JBM. Given a coupling $(X,\hat X)$, we define the \textit{coupling time} by
\begin{equation*} \label{eq:coupling_time_defn}
	\Tc = \inf\{t\ge 0\,:\, X_s = \hat X_s \,\text{ for all } s\ge t \} \,.
\end{equation*}
Recall that the tail distribution of any coupling time provides an upper bound on the total variation distance between the laws of $X_t$ and $\hat X_t$ via the coupling inequality \citep{aldousRandomWalksFinite1983}:
\begin{equation} \label{eq:coupling_inequality}
	\tv{\law(X_t) - \law(\hat X_t)} \le \prob{\Tc > t} \,,
\end{equation}
where 
\begin{equation}\label{eq:TV_defn}
	\tv{\law(X_t) - \law(\hat X_t)} = \sup_A \{ \prob{X_t\in A} - \mathbb{P}(\hat X_t\in A) \}
\end{equation}
and where the supremum is taken over all Borel subsets of \interval.
A coupling is called \textit{successful} if $\prob{\Tc > t}\to 0$ as $t\to \infty$; it is called \textit{maximal} if it achieves equality in \eqref{eq:coupling_inequality} for all times $t$. 
 
 It is well known that a maximal coupling of two random processes exists under extremely mild conditions: see  \cite{Griffeath-1975}, \cite{Pitman1976} and \cite{goldsteinMaximalCoupling1979} for discrete-time processes, and \cite{sverchkovMaximalCouplingDvalued1990} for the case of c\`adl\`ag processes with Polish state-space. However, in most cases, explicit construction of a maximal coupling is extremely difficult, and it is natural for attention to focus on classes of couplings which are more readily realisable. One such class is that of \emph{co-adapted} couplings:

\begin{defn} \label{def:co-adapted}
	A coupling $(X,\hat X)$ is called \textit{co-adapted} if there exists a filtration $(\mathcal{F}_t)_{t\ge 0}$ such that $X$ and $\hat X$ are both adapted to $(\mathcal{F}_t)_{t\ge 0}$ and for any $0\le s\le t$,
		\[ \law(X_t\,|\,\mathcal{F}_s) = \law(X_t\,|\,X_s) \quad\text{and} \quad \law(\hat X_t\,|\,\mathcal{F}_s) = \law(\hat X_t\,|\,\hat X_s) \,. \]
\end{defn}
In other words, both $X$ and $\hat X$ are Markov with respect to the filtration $(\mathcal{F}_t)_{t\ge0}$. \cite{kendall2015coupling} refers to co-adapted couplings as \textit{immersed}, since the condition of Definition~\ref{def:co-adapted} is equivalent to demanding that the natural filtrations of $X$ and $\hat X$ are both immersed in a common filtration.

Even though maximal couplings certainly need not be co-adapted, there are a few processes for which maximal co-adapted couplings have been shown to exist. Probably the simplest of these is the reflection/mirror coupling of Euclidean Brownian motions, in which the path of one process, until the coupling time, is obtained by reflecting the other in the hyperplane bisecting the line joining their starting points \citep{LindvallRogers-1986}; indeed, this is the unique maximal co-adapted coupling \citep{Hsu2013a}, a result which holds more generally for Brownian motion on a Riemannian manifold \citep{Kuwada2007}. More recently, \cite{banerjeeRigidityMarkovianMaximal2017} showed that a maximal co-adapted coupling for smooth elliptic diffusions on a complete Riemannian manifold can only possibly exist if the underlying space is a sphere, Euclidean space or hyperbolic space.


Of particular relevance to the work of the current paper is the analysis of a symmetric random walk on the hypercube performed by \cite{Connor2008}, and subsequently generalised by \cite{Connor-2013}. 
They considered the class of co-adapted couplings for two such random walks, and showed that there exists a \emph{stochastically optimal} coupling within this class. In other words, they exhibited a coupling whose coupling time $T^*$ satisfies, simultaneously for all $t\ge 0$,
\[ \prob{T^*>t} = \min\set{ \prob{T>t}\,:\, \text{co-adapted coupling times $T$}}\,. \]
Furthermore, they showed that this optimal co-adapted coupling does not achieve equality in \eqref{eq:coupling_inequality}, thus demonstrating that there does not exist a maximal co-adapted coupling for this random walk. A result in the same vein was proved by \cite{kendall2015coupling}, who showed that there exists a stochastically optimal co-adapted coupling for the two-dimensional process consisting of Brownian motion together with its local time at 0; numerical evidence indicates that this coupling is once again not maximal. We also highlight here the paper of \cite{jackaMirrorSynchronousCouplings2014}, which investigated (amongst other things) whether the reflection coupling for Brownian motions optimises various functions of the coupling time of the corresponding geometric Brownian motions. They showed that the Laplace transform of the coupling time is maximised by the reflection coupling, but that whether this coupling also solves the finite time horizon problem (to minimise the coupling time's tail distribution) depends upon an underlying drift parameter.

By way of motivation for the study of \JBM, consider the following problem. Suppose that $X$ and $\hat X$ start from opposite points on the circle (i.e. $|X_0-\hat X_0| = \pi$, where $|\cdot|$ denotes the shortest distance between two points on the circumference of the unit circle) and that we wish to couple their evolution so as to minimise (some non-decreasing function of) the corresponding coupling time. Write $D_t = |X_t-\hat X_t| \in [0,\pi]$. Consider the two following cases:
\begin{enumerate}
	\item If $\lambda = 0$ then there are no jumps; in this case the fastest coupling is achieved by setting $\hat X_t = \pi - X_t\ (\mathrm{mod}\ 2\pi)$, i.e. having the two processes \textit{reflect} about either of the two points at distance $\pi/2$ from both $X_0$ and $\hat X_0$. The resulting process $D_t$ is a reflected Brownian motion on $[0,\pi]$, and the corresponding  coupling time is given by $\Tc = \inf\{t\,:\, D_t = 0\} = \tau$, where $\tau$ is the first hitting time of the set $\{-\pi/2,\pi/2\}$ by the Brownian motion $\tfrac12 B_t$ making up the diffusion component of $X$.
	
	\item Alternatively, if $\lambda$ is large then we expect jumps to be occurring quickly; in this case it may be better to \textit{synchronise} the driving Brownian motions, and to let the Poisson processes $N$ and $\hat N$ be independent until the first time that we see an incident on either one. In other words, we maintain $D_t=\pi$ for all $t<J$, where $J\sim\expdist(2\lambda)$ is the first time that either $X$ or $\hat X$ jumps, and then the coupling time is given by $\Tc=J$.
\end{enumerate}
Note that the ways in which the driving Brownian motions are coupled in these two scenarios are as different as can be; the first coupling proceeds by maximising the volatility of $B_t-\hat B_t$ for all $t<\Tc$, whilst the second minimises it (setting it to zero, in fact). Taking this intuition further, it is not unreasonable to suppose that there might be a critical value of $\lambda$ with the property that the second of the two couplings sketched above is better than the other if and only if the jump rate exceeds this critical value. 

Furthermore, it is clear that when $\lambda=0$ the mirror coupling described in case 1 above will be \emph{maximal}, but that this will not be true of the coupling strategy outlined in case 2: if one chooses to synchronise the driving Brownian motions and couple at time $J$, then there is a positive probability that coupling would have occurred faster if the Brownian motions had instead been reflected. Indeed, we shall show in Section~\ref{sec:maximal} that when $|X_0-\hat X_0| = \pi$ the total variation distance between the distributions of $X_t$ and $\hat X_t$ satisfies
\[ 
	\tv{\law(X_t\,|\,X_0=0) - \law(\hat X_t\,|\,\hat X_0=\pi)} = \prob{\min\set{J, \tau} > t}\,,
\]
where $J$ and $\tau$ are independent. We shall use this observation to explicitly describe a maximal, but non co-adapted, coupling of two \JBM\ processes started from antipodal points of the circle.

In Section~\ref{sec:co-adapted} we restrict attention to the class of co-adapted couplings of two \JBM\ processes with arbitrary starting points, and consider two couplings motivated by the cases outlined above. We shall refer to these as \emph{Mirror} and \emph{Synchronous} couplings, but these names relate \emph{solely} to the way in which the driving Brownian motions are coupled when $D_t = \pi$. A formal definition will be given later (Definition~\ref{def:formal_RoW}), but in order to state our main results we provide an informal description here.

\newpage

\begin{defn}[Informal description of Mirror and Synchronous couplings]\label{def:couplings-informal}
	The Mirror and Synchronous couplings treat the jump and diffusion components of $X$ and $\hat X$ as follows.
	\begin{description}
		\item[\phantom{J}Jumps:] \emph{  }
		\begin{itemize}
			\item if $D_t \in(0,\pi/2]$, \textit{synchronise} the driving Poisson processes (i.e. make $\hat X$ jump at time $t$ if and only if $X$ sees a jump at time $t$);
			\item if $D_t \in(\pi/2,\pi]$, make the driving Poisson processes \textit{independent}.
		\end{itemize}
		
		\item[\phantom{J}Diffusion:] \emph{  }
		\begin{itemize}
			\item if $D_t \in(0,\pi)$, \textit{reflect} the driving Brownian motions (i.e. set $\hat B_t = - B_t$);
			\item if $D_t=\pi$, then:
			\begin{itemize}
				\item for the \textbf{Mirror coupling}, \textit{reflect} the driving Brownian motions (i.e. continue to set $\hat B_t = - B_t$);
				\item for the \textbf{Synchronous coupling},  \textit{synchronise} the driving Brownian motions (i.e. set $\hat B_t =   B_t$).
			\end{itemize}
		\end{itemize}
	\end{description}
\end{defn}
\medskip
Trivially, the Mirror coupling is successful for all jump rates $\lambda\ge 0$, while the Synchronous coupling is successful for any $\lambda>0$. We note that our Synchronous coupling has a similar two-stage approach to the `reflection/synchronised coupling' of Brownian motion together with local time described in~\cite{kendall2015coupling}. In both cases the Brownian motions are reflected until some stopping time, $T_1$ (in our case, the first hitting time of $D_t$ on the set $\set{0,\pi}$); if coupling has not occurred at time $T_1$ then the Brownian motions are subsequently synchronised until the coupling time, $T$.

In Section~\ref{sec:co-adapted} we shall calculate the Laplace transform of the associated coupling times, $\Tm$ (Mirror) and $\Ts$ (Synchronous), and then use ideas from stochastic control to prove that the Synchronous coupling is faster than the Mirror coupling if and only if $\lambda\ge \lambda^\star$, where $\lambda^\star=0.08337$ is the unique solution to the equation
\begin{equation}\label{eq:def_lambda*}
\cosech(\pi\sqrt{\lambda}) + 2\pi\sqrt{\lambda} = 2\coth(\pi\sqrt{\lambda})\,.
\end{equation}
More precisely, for this range of $\lambda$ the Synchronous coupling turns out to uniquely maximise the Laplace transform of the coupling time within the class of all co-adapted couplings.
\begin{thm}\label{thm:main_thm_LT}
	For any $\lambda\ge \lambda^\star$ the Synchronous coupling uniquely maximises, within the class of co-adapted couplings, the Laplace transform of the corresponding coupling time. That is, for any $\lambda\ge\lambda^\star$ and any $x\in(0,\pi]$, the coupling time $\Ts$ maximises $\Ex{e^{-\gamma T}\,|\, D_0=x}$ simultaneously for all $\gamma\ge 0$. 
\end{thm}
It does not appear that there is \emph{any} co-adapted coupling which simultaneously maximises the Laplace transform for all values of $\gamma$ when $\lambda<\lambda^\star$ (see Remark~\ref{rem:no_LT_optimal}). However, the Mirror coupling does turn out to minimise the \textit{mean} coupling time for this set of jump rates.
\begin{thm}\label{thm:main_thm_mean}
	For any $\lambda\in[0,\lambda^\star)$, the Mirror coupling minimises the mean coupling time within the class of co-adapted couplings. That is, for any $\lambda\in[0,\lambda^\star)$ and any $x\in(0,\pi]$, the coupling time $\Tm$ minimises $\Ex{T\,|\, D_0=x}$ within the class of co-adapted couplings.
\end{thm}
In the sequel we shall use the terms `LT-optimal' and `mean-optimal', and shall usually refrain from writing `within the class of co-adapted couplings'. Note that when $\lambda=0$ (no jumps) the Mirror coupling reduces to the reflection coupling for Brownian motion and is therefore \textit{maximal}. Similarly, when $\lambda\to \infty$ the Synchronous coupling makes jumps from $x\mapsto\pi-x$ occur immediately for any $x>\pi/2$, and so in the limit the difference process $D$  behaves like reflected Brownian motion on $[0,\pi/2]$; it is simple to see that this will once again be maximal. Table~\ref{tab:summary} summarises the non-trivial way in which the various optimality properties of our coupling strategies depend upon the jump rate. 

\begin{table}[hbt]
	\centering
	\renewcommand{\arraystretch}{1.2}
	\begin{tabular}{|c|l|}
		\hline
		Jump rate & Coupling properties \\
		\hline
		$\lambda=0$ & {Mirror is maximal} \\
		$0<\lambda<\lambda^\star$ & Mirror is mean-optimal; no LT-optimal coupling exists \\
		$\lambda^\star\le \lambda$ & {Synchronous is LT-optimal (and hence also mean-optimal)} \\
		$\lambda\to \infty$ & $\Ts$ converges to the maximal coupling time\\
		\hline
	\end{tabular}
	\caption{Summary of optimal coupling properties for different values of the jump rate $\lambda$.}
	\label{tab:summary}
\end{table}
\JBM\ therefore makes an interesting addition to the relatively small number of examples in the literature of processes for which precise results about optimal co-adapted couplings have been established. In particular, to the best of the authors' knowledge, our results are the first of their kind for a L\'evy process with both continuous and jump components.

\begin{rem}
	Our original proofs of Theorems~\ref{thm:main_thm_LT} and \ref{thm:main_thm_mean} were based on excursion theory of Brownian motion, using calculations similar to those in \citet[\S VI.56]{rogers2000diffusions}. We are grateful to an anonymous referee who suggested the more direct method used in Section~\ref{sec:co-adapted} below. However, we still find the excursion approach very appealing, and the corresponding calculations can be found in the original arXiv version of this paper~\citep{Connor.Merli-2022}. Further details can also be found in the PhD thesis of the second author~\citep{Merli2021}, alongside supporting evidence for the correctness of our results obtained via direct simulation of the two couplings.
\end{rem}

\section{Maximal coupling}\label{sec:maximal}

In this section we briefly describe a construction of a maximal, but non-co-adapted, coupling in the case that $D_0 = \pi$ (i.e. the two coupled \JBM\ processes begin at opposite points of the circle). Let $(X_0,\hat X_0) = (0,\pi)$. For any set $A\subset \interval$, 
\[ 
\prob{X_t \in A} = \sum_{k=-\infty}^\infty \prob{\tfrac12 B_t \in A +  2k\pi \,, N^\lambda(t) \text{ even}}+ \prob{\tfrac12 B_t \in A +  (2k+1)\pi \,, N^\lambda(t) \text{ odd}}\,.
\]
Since 
\[ 
\mathbb{P}(N^\lambda(t) \text{ even}) = \frac{1+e^{-2\lambda t}}{2} \ge \mathbb{P}(N^\lambda(t) \text{ odd})\,,
\] 
it is clear that the set on which $X_t$ has greater density than $\hat X_t$ is, for all $t\ge 0$, the half-circle $A^\star=(-\pi/2,\pi/2)$ centred at $X_0$. As above, let $\tau$ denote the first hitting time of the set $\{-\pi/2,\pi/2\}$ by the Brownian motion $\tfrac12 B_t$. By the definition of total variation distance \eqref{eq:TV_defn}, we see that 
\begin{align}
	\tv{\law(X_t\,|\, X_0=0) - \law(\hat X_t\,|\, \hat X_0 = \pi)} &= \prob{X_t\in A^\star} - \mathbb P({\hat X_t\in A^\star}) \nonumber \\ 
	&= 2\prob{X_t\in A^\star} - 1 \nonumber \\
&= e^{-2\lambda t} \bra{ 2\prob{\tfrac12 B_t \in A^\star +  2\pi\Z} - 1} \nonumber \\
&= e^{-2\lambda t} \prob{\tau > t}\,, \label{eq:maxCouplingTail}
\end{align}
where this final equality follows by the reflection principle. Thus we see that when $X$ and $\hat X$ begin from antipodal points, the total variation distance between their laws is given by the tail distribution of the random variable $\min\set{J, \tau}$, where $J\sim\expdist(2\lambda)$. 

This observation leads to an explicit construction of a maximal coupling, as follows. Let $N^{2\lambda}$ be a marked Poisson process of rate $2\lambda$, whose incident times are denoted by $J_k$, and with marks $Y_k\iid\berdist(1/2)$. Let $B$ be a standard Brownian motion started at 0, with hitting time $\tau$ defined as above. We first of all use these to define the \JBM\ process $X$: the diffusion component of $X$ is given by $\tfrac12 B$, and $X$ makes a jump of size $\pi$ at incident times $J_k$ if and only if the corresponding mark $Y_k$ equals 1.

We now define $\hat X$, started at $\pi$, in such a way that $\hat X$ and $X$ couple almost surely at time $T^\star=\min\set{J_1,\tau}$:
\begin{itemize}
	\item On the event $\set{\tau< J_1}$, we let the diffusion component of $\hat X$ equal $-\tfrac12 B$, so that $\hat X$ is the reflection of $X$ about the points $\pm\pi/2$. In this case $\hat X$ and $X$ will meet at time $\tau$.
	
	\item On the event $\set{J_1<\tau}$, we let the diffusion component of $\hat X$ equal $\tfrac12 B$ (i.e. $\hat X$ and $X$ move synchronously, remaining at distance $\pi$ from each other) until time $J_1$. We then make $\hat X$ jump by $\pi$ at time $J_1$ if and only if the mark $Y_1$ equals 0. That is, $\hat X$ jumps at time $J_1$ if and only if $X$ does \emph{not} jump at that time, and since their diffusion components were synchronised, $\hat X$ and $X$ will meet exactly when one of them jumps at time $J_1$.
	
	\item From time $\min\set{J_1,\tau}$ onwards, the diffusion and jump components of $\hat X$ equal those of $X$.
\end{itemize}
This is a valid coupling (the process $\hat X$, viewed marginally, really is a \JBM\ process started from $\pi$). Furthermore, it is clearly maximal but not co-adapted: the evolution of $X$ is adapted to the natural filtration generated by $\set{N^{2\lambda},B}$, whereas the evolution of $\hat X$ until the coupling time depends on knowledge of which of the times $J_1$ and $\tau$ occurs first.

\begin{rem}
	For general starting states ($D_0=x<\pi$) such an explicit description of a maximal coupling is significantly more challenging. The main complication is that it is no longer true that the set on which $X_t$ has greater density than $\hat X_t$ is independent of $t$. Rather, with $(X_0,\hat X_0) = (-x/2,x/2)$, the set $A^\star_t$ on which $X_t$ has greater density takes the form 
\[ A^\star_t = (-r_t,0)\cup(r_t,\pi) \,, \]
where $r_t=r_t(x)$ is the unique point in $(0,\pi)$ at which the densities of $X_t$ and $\hat X_t$ agree. Depending upon the value of $\lambda$, $r_t$ tends to either $\pi/2$ or $\pi$ as $t\to\infty$, but not necessarily monotonically. The only exception to this rule is in the limit as $\lambda\to \infty$, when $r_t \to \pi/2$ for all values of $t$ (and for all $x<\pi$); in this case $A_t$ does not depend upon $t$, and an explicit maximal coupling is once again straightforward to describe. (See the comment at the end of Section~\ref{sec:intro}.)
\end{rem}

\section{Co-adapted couplings}\label{sec:co-adapted}

We begin the search for optimal co-adapted couplings by adopting the perspective of stochastic control. We wish to work with a pair of \JBM\ processes which are adapted to a common filtration of $\sigma$-algebras $(\mathcal{F}_t)$; it is more convenient to work initially with jumpy Brownian motions $(X,\hat X)$ on the real line (i.e. processes satisfying equation~\eqref{eq:JBM_defn} without the `mod $2\pi$') and then project to the circle. Let $(\mathcal{F}_t)_{t\ge 0}$ be any filtration to which the following {independent} random processes are all adapted: 
two standard Brownian motions, $B$ and $\tilde B$;
two marked Poisson processes $N^\lambda$ and $\tilde N^\lambda$ of rate $\lambda\ge 0$, whose marks ($U_t$ and $\tilde U_t$, respectively) are independent and identically distributed as $\unifdist[0,1]$.

First, note that any co-adapted coupling of two Brownian motions $B$ and $\hat B$ can be represented by the SDE
\[ \dd \hat B_t = \theta_t \,\dd B_t + \sqrt{1-\theta_t^2}\, \dd \tilde B_t\,, \]
where $\theta_t$ is a predictable random process taking values in $[-1,1]$ \citep[Lemma 6]{kendall2009brownian}. Similarly, as explained in~\citep{Connor2008}, any co-adapted coupling of two Poisson processes $(N^\lambda,\hat N^\lambda)$ can be written as 
\begin{equation}\label{eqn:poisson_couple}
	\hat N^\lambda (\dd t) = \Indicator{U_t\le p_{t-}} N^\lambda(\dd t) + \Indicator{\tilde U_t> p_{t-}} \tilde N^\lambda(\dd t) \,, 
\end{equation}
where $p_t$ is a c\`adl\`ag control process adapted to $(\mathcal{F}_t)$, taking values in $[0,1]$. (More accurately, \cite{Connor2008} describe any co-adapted coupling of two unit-rate random walks on the hypercube, $\Z_2^n$, in terms of $(n+1)^2$ independent marked Poisson processes which are controlled by a doubly stochastic matrix-valued process; equation~\eqref{eqn:poisson_couple} is just a simplified parametrisation of their result when $n=1$.)

Combining these two results, it is clear that the joint process $(X_t,\hat X_t)$ satisfies
\begin{equation}\label{eq:joint}
\begin{pmatrix}
	\dd X_t\\
	\dd \hat X_t
\end{pmatrix} = \frac12\, 
\begin{pmatrix}
1 & 0 \\
\theta_t &\sqrt{1-\theta_t^2}
\end{pmatrix}
\begin{pmatrix}
	\dd B_t\\
\dd \tilde B_t
\end{pmatrix}
+ \pi \,
\begin{pmatrix}
1 & 0 \\
\Indicator{U_t\le p_{t-}} & \Indicator{\tilde U_t> p_{t-}}
\end{pmatrix}
\begin{pmatrix}
  N^\lambda (\dd t) \\
 \tilde N^\lambda (\dd t)
\end{pmatrix}\,,
\end{equation}
where the control process $c_t = (\theta_t,p_t)$ takes values in $[-1,1]\times [0,1]$ and is adapted to the filtration $(\mathcal{F}_t)$. The control process explicitly determines the dependence between $X$ and $\hat X$. In particular, if $\theta_t=0$ then their continuous components are independent; if $\theta_t = 1$ then they are \textit{synchronised}; and if $\theta_t = -1$ then they are \textit{mirror} or \textit{reflection} coupled. Similarly, if $p_t=1$ then the jump components of $X$ and $\hat X$ are synchronised, whereas if $p_t=0$ then they are independent.

In what follows, we will be interested in the difference process $Z_t = X_t-\hat X_t$, and in particular the time it takes for this to hit a multiple of $2\pi$ (at which time the projections of the two processes onto the unit circle will meet). Using~\eqref{eq:joint} we see that
\[
\dd Z_t = \frac12\bra{(1-\theta_t)\dd B_t -  \sqrt{1-\theta_t^2}\,\dd \tilde B_t	}
+ \pi \bra{ (1-\Indicator{U_t\le p_{t-}}) N^\lambda (\dd t)
	- \Indicator{\tilde U_t> p_{t-}} \tilde N^\lambda (\dd t)}\,.
\]
Thus $Z$ has the same dynamics as a L\'evy process on $\R$ whose continuous component has volatility $(1-\theta)/2$ and which makes jumps of size $+\pi$ and $-\pi$ each at rate $\lambda(1-p)$. For any $z\in\R$ the distances of $z+\pi$ and $z-\pi$ from the set $2\pi\Z$ are equal, and so the distribution of the time taken for $Z$ to hit the set $2\pi\Z$ is unchanged if we alter the dynamics to make all jumps of size $+\pi$. Furthermore, the independence of $N$ and $\tilde N$ means that  
\[  (1-\Indicator{U_t\le p_{t-}}) N^\lambda (\dd t)
+ \Indicator{\tilde U_t> p_{t-}} \tilde N^\lambda (\dd t) 
\,\,\eqdist \,\,\Indicator{U'_t> p_{t-}} N^{2\lambda}(\dd t) \,,
\]
where $\eqdist$ denotes equality in distribution, and where $N^{2\lambda}$ is a marked Poisson process of rate $2\lambda$ with marks $U'\sim\unifdist[0,1]$. Thus it suffices, for any given adapted control process $c_t = (\theta_t,p_t)$, to consider the hitting time on the set $2\pi\Z$ of the process given by 
\[
\dd Z_t = \sqrt{\frac{1-\theta_t}{2}}\,\dd B_t+\pi \Indicator{U'_t> p_{t-}} N^{2\lambda}(\dd t)\,; \qquad Z_0 = X_0-\hat X_0 \,.
\]

As in Section~\ref{sec:intro}, we shall work in the sequel with the process $D_t=|Z_t-2\pi\Z|\in[0,\pi]$, which measures the distance between the projections of $X$ and $\hat X$ on the circumference of the unit circle. Until the coupling time $T=\inf\{t\,:\,D_t=0\}$, $D$ behaves like a \textit{reflected} Brownian motion (with volatility controlled by $\theta$), and with a jump (controlled by $p$) at time $t$ having the effect of making $D_t = \pi -D_{t-}$. (That is, a jump of size $\pi$ in $Z$ results in $D$ being reflected about the point $\pi/2$.)

This allows us to view the search for an optimal co-adapted coupling as a stochastic control problem. For a given value function, we seek a control process $c_t=(\theta_t,p_t) \in [-1,1]\times[0,1]$ such that the time taken for the corresponding difference process $D$ to hit zero minimises/maximises the value function, as appropriate. 
As noted in~\cite{oksendal2007applied}, it suffices to restrict attention to Markov controls of the form $c_t = c(D_{t-})$.


Using this setup we can now give a precise definition of the Mirror and Synchronous couplings for two \JBM\ processes (cf Definition~\ref{def:couplings-informal}).
\begin{defn}\label{def:formal_RoW} 
	Suppose $D_{t-}=x\in[0,\pi]$. The Mirror coupling is the co-adapted coupling with control $c\mirror = (\theta\mirror,p\mirror)$ at time $t$ given by
\[
\text{Diffusion:} \quad\theta\mirror(x) = 2\Indicator{x=0} -1\,; \qquad
\text{Jumps:} \quad p\mirror(x) = \Indicator{x\le\pi/2} \,.
\]
Similarly, the Synchronous coupling uses control $c\sync = (\theta\sync,p\sync)$, where:
\[
\text{Diffusion:} \quad\theta\sync(x) = 2\Indicator{x=0\text{ or }x=\pi} -1\,;  \qquad
\text{Jumps:} \quad p\sync(x) = \Indicator{x\le\pi/2} \,.
\]
\end{defn}
As explained above, $D$ behaves like a Brownian motion on $(0,\pi)$, for which we can control the speed using $\theta$, and if $D$ sees a jump then this simply reflects it around the point $\pi/2$. Since we wish to minimise the time taken for $D$ to hit 0, it seems intuitively sensible to maximise the speed of the diffusion when $D_t\in(0,\pi)$, and to maximise/minimise the jump rate according to whether or not a jump would reduce the value of $D_t$; this is exactly what both Mirror and Synchronous couplings achieve. The only difference between the two couplings is in the choice of $\theta(\pi)$: $\theta\mirror(\pi) = -1$ forces the Brownian component to reflect downwards from the barrier at $\pi$, whereas $\theta\sync(\pi) = 1$ means that $D$ waits an $\expdist(2\lambda)$ amount of time at level $\pi$ before jumping directly to 0.

\subsection{LT-optimality}\label{ssec:LT-optimality}

In this section we shall prove the LT-optimality of the Synchronous coupling when $\lambda\ge\lambda^\star$, as claimed in Theorem~\ref{thm:main_thm_LT}. We begin by deriving explicit formulas for the Laplace transforms of our two co-adapted coupling strategies.

Given the jump rate $\lambda$ and a constant $\gamma\ge 0$, we shall write
\begin{equation*}\label{eq:alpha_beta_def}
	\alpha =\sqrt{2(2\lambda+\gamma)}\,, \quad\text{and}\quad  \beta=\sqrt{2\gamma}\,, 
\end{equation*}
and use these to define the following four non-negative constants (depending implicitly upon $\lambda$ and $\gamma$):
\begin{align*}
	\kappa_1 &= \frac{2 \cosh(\alpha\tfrac{\pi}{2})\sinh(\beta\tfrac{\pi}{2})}{2\cosh(\alpha\tfrac{\pi}{2})\cosh(\beta\tfrac{\pi}{2})-1}  \,; 
	&
	\kappa_2 &= \tfrac{\beta}{\alpha} \kappa_1 \sech(\alpha\tfrac{\pi}{2})\,; \nonumber \\
	\kappa_3 &=  \sech(\beta\tfrac{\pi}{2})\left( \tfrac{\beta}{2\alpha}\cosech(\alpha\tfrac{\pi}{2})+\sinh(\beta\tfrac{\pi}{2})\right)\,; 
	&
	\kappa_4 &=  \bra{\tfrac{\beta}{\alpha}}^2\cosech(\alpha\tfrac{\pi}{2}) \,. \label{eq:kappa_defns} 
\end{align*}

\newpage
\begin{lem} \label{lem:LT_for_RoW}
The Laplace transforms for the coupling times $\Tm$ and $\Ts$ under the Mirror and Synchronous couplings started from distance $x\in[0,\pi]$, are given by the following formulas.

\noindent Mirror coupling:
\[ 
\Ex[x]{e^{-\gamma \Tm}} = w\mirror(x) := \begin{cases}
\cosh(\beta x) - \kappa_1 \sinh(\beta x) & \quad 0\le x\le\tfrac{\pi}{2} \\
\Ex[\pi-x]{e^{-\gamma \Tm}}-\kappa_2
\sinh(\alpha(x-\tfrac{\pi}{2})) & \quad \tfrac{\pi}{2}<x\le\pi \,.
\end{cases}
\]
%
Synchronous coupling:
\[ 
	\Ex[x]{e^{-\gamma \Ts}} = w\sync(x) :=\begin{cases}
		\cosh(\beta x) - \kappa_3 \sinh(\beta x) & \quad 0\le x\le\tfrac{\pi}{2} \\
		\Ex[\pi-x]{e^{-\gamma \Ts}} - \kappa_4\sinh(\alpha(x-\tfrac{\pi}{2})) & \quad \tfrac{\pi}{2}<x\le\pi \,.
		\end{cases}
\]
\end{lem}

\begin{rem}\label{rem:x=pi}
	Note that when $x=\pi$, the final formula in Lemma~\ref{lem:LT_for_RoW} simplifies to
	\[ \Ex[\pi]{e^{-\gamma \Ts}} = 1-\bra{\frac{\beta}{\alpha}}^2= \frac{2\lambda}{2\lambda+\gamma}\,; \]
	this is just the Laplace transform of an $\expdist(2\lambda)$ distribution, which is consistent with the observation that under the Synchronous coupling the two processes will meet precisely when one of them jumps to the other side of the circle. Similarly, the formula for the Mirror coupling when $\lambda=0$ and $x=\pi$ reduces to 
	\[ \Ex[\pi]{e^{-\gamma \Tm}} = \sech(\pi\sqrt{2\gamma})\,, \]
	which is the Laplace transform of the time taken for a standard Brownian motion to hit $\pm \pi$ (and hence of the time taken for a Brownian motion started from, and reflected at, $\pi$ to hit 0).
\end{rem}

Before proceeding any further, it will be helpful to quickly establish some basic properties of the function $w\sync$ defined in Lemma~\ref{lem:LT_for_RoW}.

\begin{lem}
	\label{lem:convexity}
	For any fixed $\gamma\ge 0$,  $w\sync(x)$ is a convex function of $x\in[0,\pi]$, and $w\sync''(\pi)=0$. 
\end{lem}
\begin{proof}
	It is clear from the definition of $w\sync$ that this function is continuous on $[0,\pi]$ and $C^2$ on $(0,\pi)$, with second derivative given by
	\[
	w\sync''(x) = \begin{cases}
		\beta^2 w\sync(x) & 0\le x\le \pi/2 \\
		\beta^2 w\sync(\pi-x) - \alpha^2 \kappa_4 \sinh(\alpha(x-\tfrac{\pi}{2})) & \pi/2 < x\le \pi\,.
	\end{cases}
	\]
	This is clearly non-negative for $x\le \pi/2$, and zero for $x=\pi$. For $x>\pi/2$, inserting the formula for $w\sync(\pi-x)$   and expanding shows that
	\begin{equation}\label{eq:second_deriv}
		w\sync''(x)  \propto s_\beta c_\beta\cosh(\beta(\pi-x)) - \left(\frac{\beta s_\beta}{2\alpha s_\alpha} + s_\beta^2 \right) \sinh(\beta(\pi-x)) - \frac{ s_\beta c_\beta}{s_\alpha} \sinh(\alpha(x-\tfrac{\pi}{2}))\,,
	\end{equation}
	where we have written $c_u = \cosh(u\tfrac{\pi}{2})$ and $s_u = \sinh(u\tfrac{\pi}{2})$. Now note that $u\mapsto us_u$ is an increasing function, and that $u\mapsto \sinh(u(x-\tfrac{\pi}{2}))/s_u$ is decreasing for any fixed $x\in[\pi/2,\pi]$. Thus (recalling that $\beta\le \alpha$) the right-hand side of \eqref{eq:second_deriv} can be bounded below by
	\[
	s_\beta c_\beta\cosh(\beta(\pi-x)) - (\tfrac12 + s_\beta^2) \sinh(\beta(\pi-x)) - c_\beta \sinh(\beta(x-\tfrac{\pi}{2})) =  \tfrac12 \sinh(\beta (\pi-x)) \ge 0\,.
	\]
\end{proof}

\begin{proof}[Proof of Lemma~\ref{lem:LT_for_RoW}]
	We shall deal first with the Synchronous coupling. It is clear that $w\sync(0) = 1$, and we have already observed that $w\sync(\pi) = 2\lambda/(2\lambda+\gamma)$, so $w\sync$ certainly takes the correct values at the two boundary points. Given $D_0 = x\in[0,\pi]$, consider the (bounded) process $W_t(x) = e^{-\gamma t} w\sync(D_t)$. We shall show that $W$ is a martingale until time $T\sync$, and then the optional stopping theorem will quickly yield that
	\[ \Ex[x]{e^{-\gamma \Ts}} = \Ex{W_{\Ts}(x)} = W_0(x) = w\sync(x) \,,\]
	as required.
	
	The martingale property is trivial when $x=\pi$. For $x\in(0,\pi)$, we apply It\^o's formula to $W_t(x)$ to obtain
	\begin{align}
		e^{\gamma t}\dd W_t(x)  &= \left(-\gamma w\sync(D_t) + \frac{1}{2}w\sync''(D_t)\right) \dd t + w\sync'(D_t)\dd D_t \nonumber \\
		&= \left(-\gamma w\sync(D_t) + \frac12 w\sync''(D_t)\right) \dd t +\dd Q_t +  \Indicator{D_t> \pi/2}\set{w\sync(\pi-D_t) - w\sync(D_t)}   N^{2\lambda}(\dd t)\,,
		\label{eq:Ito1}
	\end{align}
	where $Q_t$ is a martingale. Using the formula for $w\sync''$ established in the proof of Lemma~\ref{lem:convexity}, and recalling that $\alpha^2 = 2(2\lambda+\gamma)$ and $\beta^2 = 2\gamma$, equation~\eqref{eq:Ito1} simplifies to become
	\begin{align}
		e^{\gamma t}\dd W_t(x)  &= \dd Q_t + 
		\Indicator{D_t> \pi/2} \kappa_4 \sinh(\alpha(D_t-\tfrac{\pi}{2}))(N^{2\lambda}(\dd t)-2\lambda \dd t)\,.
		\label{eq:Ito2}
	\end{align}
	The final bracketed term on the right-hand side is of course a compensated Poisson process, and we conclude that $W$ is indeed a martingale.
	
	The Mirror coupling case is almost identical: repeating the argument above using $w\mirror$ in place of $w\sync$ shows that, when $x\in(0,\pi)$, equation~\eqref{eq:Ito2} holds with $\kappa_4$ replaced by $\kappa_2$. It only remains to check that the appropriate boundary conditions are satisfied for $w\mirror$: here we require $w\mirror(0)=1$ and $w\mirror'(\pi) = 0$, and these follow trivially from the definition of $w\mirror$.
\end{proof}

We are now in a position to complete the proof of Theorem~\ref{thm:main_thm_LT}. Let $\mathcal{C}$ denote the set of all successful co-adapted couplings. For any coupling $c\in\mathcal{C}$, and any fixed $\gamma\ge 0$, write
\[ w_c(x) = \Ex[x]{e^{-\gamma T_c}} \]
for the Laplace transform of the associated coupling time when starting from $D_0=x\in[0,\pi]$. The value function 
\[ \hat w(x) = \sup_{c\in \mathcal{C}} w_c(x) \]
	solves the HJB equation \cite[Chapter 3]{oksendal2007applied}, which may be derived here using It\^o's formula in a similar manner to equation~\eqref{eq:Ito1}. We first of all deal with the case when $x\in[0,\pi)$, for which the HJB equation is as follows:
\begin{equation}\label{eq:HJB}
	\sup_{\theta,\, p} \left( -\gamma \hat w(x) + \frac{(1-\theta(x))}{4} \hat w''(x) + 2\lambda(1-p(x))\set{\hat w(\pi-x)-\hat w(x)} \right)= 0 \,.
\end{equation}

Consider the Laplace transform for the Synchronous coupling, $w\sync(x)$. We saw in Lemma~\ref{lem:convexity} that this is convex on $[0,\pi]$; moreover, using the formula in Lemma~\ref{lem:LT_for_RoW} it may quickly be checked that  $w\sync(\pi-x)-w\sync(x)>0$ if and only if $x\in(\pi/2,\pi]$. Thus if we replace $\hat w(x)$ with $w\sync(x)$ in the left-hand side of \eqref{eq:HJB}, the corresponding supremum for $x\in[0,\pi)$ is obtained by taking $\theta(x) = -1$ and $p(x) =  \Indicator{x\le\pi/2}$; since these values agree with the control $(\theta\sync,p\sync)$ in Definition~\ref{def:formal_RoW}, it also follows that this supremum is indeed zero. Therefore the function $w\sync$ satisfies \eqref{eq:HJB} for $x\in[0,\pi)$; if we can also show that $w_c(\pi)\le w\sync(\pi)$ for all couplings $c\in\mathcal{C}$, then it will follow that $\hat w  = w\sync$, i.e. that the Synchronous coupling is LT-optimal (see Theorem 3.1 of \cite{oksendal2007applied}).

Starting from $D_0=\pi$, if $\theta_c(\pi)\neq 1$ then the diffusion component of $D$ reflects downwards off the barrier at $\pi$. This, along with the fact that $w\sync''(\pi) =0$ (Lemma~\ref{lem:convexity}), implies that the relevant quantity to be maximised over admissible controls $(\theta_c,p_c)$ becomes
\begin{equation}\label{eqn:pi_start}
	-\gamma w\sync(\pi) + 2\lambda (1-p_c(\pi)) 	\set{1- w\sync(\pi)}  - \sqrt{\frac{1-\theta_c(\pi)}{2}}w\sync'(\pi)\,.
\end{equation}
We therefore need to show that 
this expression is maximised by taking $p_c(\pi) = p\sync(\pi) = 0$ and $\theta_c(\pi) = \theta\sync(\pi) = 1$. The first of these is trivial, since $w\sync(\pi)\in(0,1)$ for all values of $\lambda$ and $\gamma$; the second will follow if and only if $w\sync'(\pi)\ge 0$ for all values of $\gamma\ge 0$. To make explicit the dependence of $w\sync'(\pi)$ on the underlying parameters $\lambda$ and $\gamma$, let us define
\begin{align*}
h(\lambda,\gamma)&:= 
w\sync'(\pi) = \beta\kappa_3 - \alpha\kappa_4\cosh(\alpha\tfrac{\pi}{2})\,.
\end{align*}
The function $h$ is continuous in both arguments, with $h(\lambda,0)=0$ and 
\[ \left. \frac{\dd h(\lambda,\gamma)}{\dd\gamma} \right\vert_{\gamma=0} = \pi + \frac{\cosech(\pi \sqrt{\lambda})}{2\sqrt{\lambda}} - \frac{\coth(\pi \sqrt{\lambda})}{\sqrt{\lambda}}\,.
\]
This derivative at $\gamma=0$ is a strictly increasing function of $\lambda$, equalling zero if and only if $\lambda=\lambda^\star$ (recall the defining equation for $\lambda^\star$, \eqref{eq:def_lambda*}). Furthermore, $h(\lambda,\gamma)$ is non-negative for all $\gamma$ if and only if $\lambda\ge\lambda^\star$; if $\lambda<\lambda^\star$ then there exist values of $\gamma>0$ for which $h$ is negative, and others for which it is positive: see Figure~\ref{fig:h_function}. It follows that $w\sync'(\pi)$ is non-negative for all $\gamma\ge 0$ exactly when $\lambda\ge \lambda^\star$, and we conclude that for this range of jump rates the Synchronous coupling is LT-optimal, as required.

\bigskip

\begin{figure}[ht]
	\centering
	\includegraphics[width=9cm]{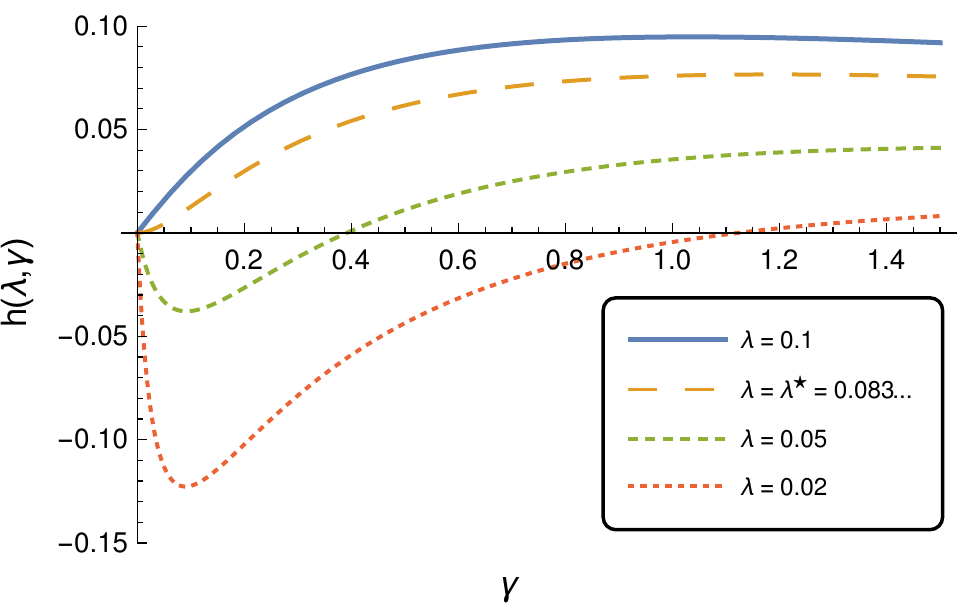}
	\caption{The function $h(\lambda,\gamma)$ for four representative values of $\lambda$. For $\lambda<\lambda^\star$ the function takes first negative then positive values as $\gamma$ increases; for $\lambda\ge \lambda^\star$ the function is non-negative for all $\gamma$.}
	\label{fig:h_function}
\end{figure}

\begin{rem}\label{rem:no_LT_optimal}
For $\lambda<\lambda^\star$, the proof given above fails only when $x=\pi$, when it is no longer true that $h(\lambda,\gamma)$ is non-negative for all $\gamma$. For such a $\lambda$, if $\gamma$ is such that $h(\lambda,\gamma)<0$, then the expression in \eqref{eqn:pi_start} is clearly maximised by setting $\theta_c(\pi) =\theta\mirror(\pi) = -1$, i.e. by using the Mirror coupling. However, any attempt to prove LT-optimality of the Mirror coupling when $\lambda<\lambda^\star$ is thwarted by the function $w''\mirror(\pi)$ taking both positive and negative values as $\gamma$ varies: as above, we arrive at the situation where for some values of $\gamma$ the optimal control is to set $\theta(\pi) =1$, and for others it is to set $\theta(\pi)=-1$. Figure~\ref{fig:nonLToptimal} shows that when $\lambda<\lambda^\star$, $w\mirror(\pi) \ge w\sync(\pi)$ when $\gamma$ is small, but that this relationship is reversed for larger $\gamma$, and so neither the Mirror nor Synchronous strategy is LT-optimal for these jump rates. This implies that when $\lambda<\lambda^\star$ there is no co-adapted coupling which maximises the Laplace transform $\Ex{e^{-\gamma T}}$ simultaneously for all values of $\gamma$. 

\begin{figure}[ht]
	\centering
	\includegraphics[width=0.48\textwidth]{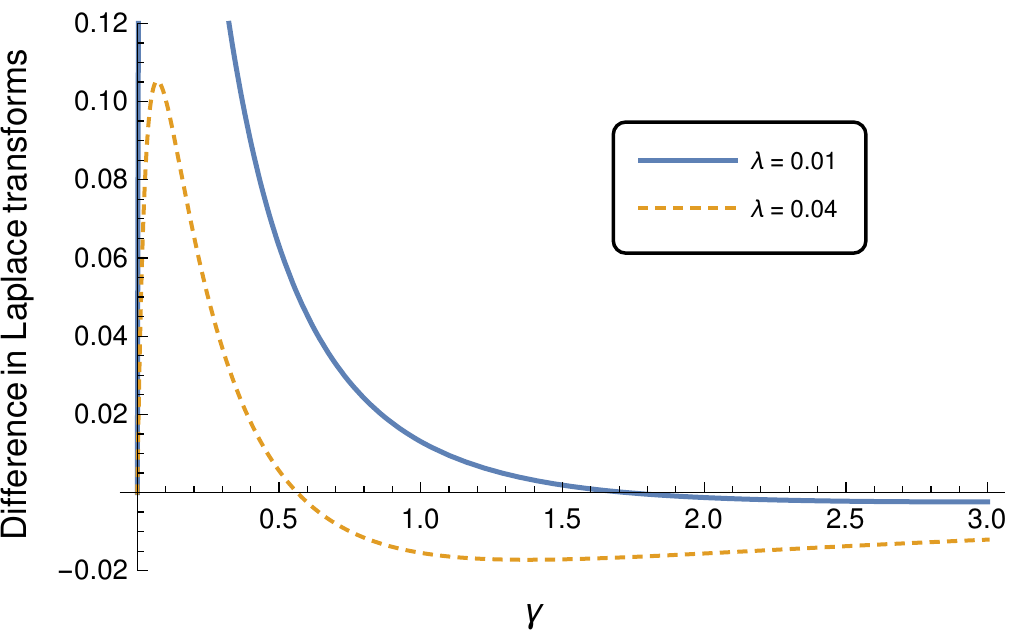}\hspace{5mm}
	\includegraphics[width=0.48\textwidth]{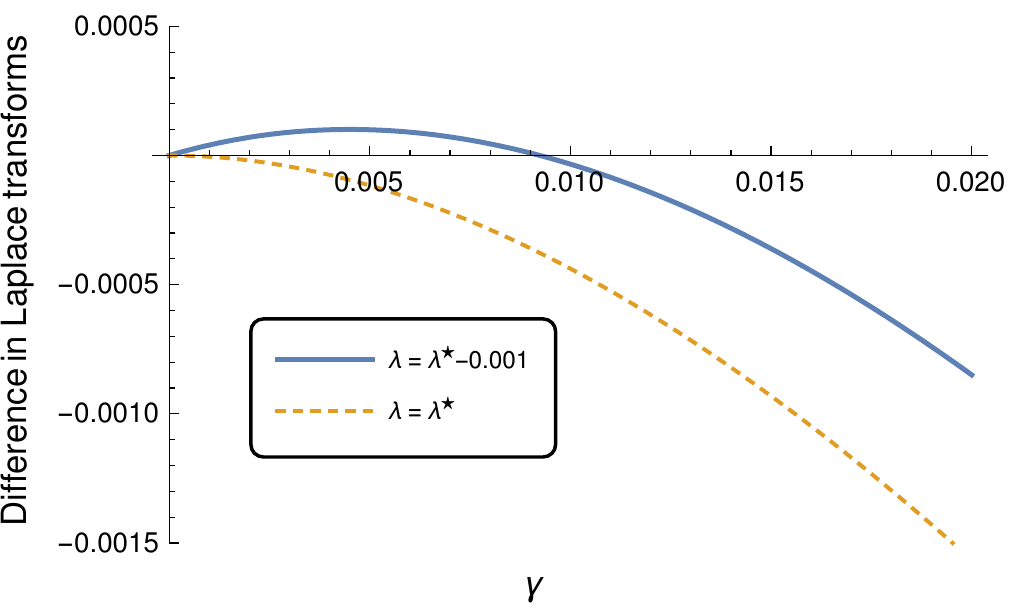}
	\caption{Difference between Laplace transforms corresponding to Mirror and Synchronous strategies: both graphs show $\mathbb{E}_{\pi}[e^{-\gamma\Tm}-e^{-\gamma\Ts}]$ 
		 as a function of $\gamma$, for different jump rates $\lambda\le\lambda^\star$. Neither coupling is LT-optimal when $\lambda<\lambda^\star$.} 
	\label{fig:nonLToptimal}
\end{figure}
\end{rem}

\subsection{Mean-optimality}\label{ssec:mean-optimality}

In this final section we compare the mean coupling times under the Mirror and Synchronous couplings. 

\begin{lem}\label{lem:expectation_RoW}
	The expectation of the coupling times under the Mirror and Synchronous couplings started from distance $x\in[0,\pi]$, are given by
	\begin{equation} 
		\Ex[x]{T}=
			\begin{cases}
				x(\pi-x)+C(\lambda)x  & \hbox{if } 0\le x\le\frac{\pi}{2} \\
				x(\pi-x)+C(\lambda)(\pi-x)+C(\lambda)\frac{\sinh(\sqrt{\lambda}(2x-\pi))}{\sqrt{\lambda}}& \hbox{if } 	\frac{\pi}{2}<x\le \pi\,,
			\end{cases} \label{eq:mean_time}
	\end{equation}
	where
	\begin{align*}\label{eq:Clambda}
	C(\lambda)= 
		\begin{cases}
	C\mirror(\lambda) = \frac{\pi}{2\cosh(\pi\sqrt{\lambda})-1} & \quad \text{if $T=\Tm$ (Mirror coupling)}\\
	C\sync(\lambda) = \frac{\cosech(\pi\sqrt{\lambda})}{2\sqrt{\lambda}} & \quad \text{if $T=\Ts$ (Synchronous coupling).} 
	\end{cases}
	\end{align*}

%
\end{lem}

\begin{proof}
	This can be proved in a couple of ways. Probably the simplest, but most tedious, is to start from Lemma~\ref{lem:LT_for_RoW} and calculate $\Ex[x]{T} = \lim_{\gamma\to0}(-\tfrac{\dd}{\dd\gamma}\Ex[x]{e^{-\gamma T}})$. Alternatively, the line of argument used in Lemma~\ref{lem:LT_for_RoW} can be followed: letting 
	$v(x)$ be given by the function in \eqref{eq:mean_time}, it is straightforward to check that this has appropriate boundary conditions and that the process	$V_t(x) = v(D_t) + t$ is a martingale until the time that $D$ first hits 0. 
\end{proof}


Note that $C(\lambda) = \min\{C\mirror(\lambda),C\sync(\lambda)\}$, and that $C\mirror(\lambda)$ and $C\sync(\lambda)$ are both positive, decreasing functions of $\lambda$, which agree precisely when $\lambda = \lambda^\star$ (recall the definition of $\lambda^\star$ in \eqref{eq:def_lambda*}). Thus, from any starting state $x\in(0,\pi]$, $\Ex[x]{\Tm}<\Ex[x]{\Ts}$ if and only if $\lambda<\lambda^\star$.

The proof of Theorem~\ref{thm:main_thm_mean}---mean-optimality of the Mirror coupling for $\lambda<\lambda^\star$---is very similar to that given above for the LT-optimality of the Synchronous coupling. In this case it suffices to show that the function $v\mirror(x) = \Ex[x]{T\mirror}$ satisfies the HJB equation:
\[
	\inf_{\theta,\, p} \bra{1+\frac{(1-\theta(x))}{4}v\mirror''(x) + 2\lambda(1-p(x))\set{v\mirror(\pi-x)-v\mirror(x)}} = 0 \,;
\]
this is straightforward, and details are omitted for the sake of brevity.

Figure~\ref{fig:expectation-comparison} gives an impression of how the mean coupling time formulas from Lemma~\ref{lem:expectation_RoW} behave as functions of the jump rate and the starting distance $x$. The left-hand plot shows that $\Ex[x]{T\mirror}<\Ex[x]{T\sync}$ for all $x\in(0,\pi]$ when $\lambda=0.05<\lambda^\star$, with this inequality being reversed for $\lambda=0.1>\lambda^\star$. The right-hand plot graphs $\Ex[x]{T\mirror}$ and $\Ex[x]{T\sync}$ as functions of $\lambda$, when $x\in\{\pi/4,\pi/2\}$: for both starting distances the intersection of the two mean coupling times at $\lambda=\lambda^\star$ is evident.

\begin{figure}[ht]
	\centering
			\includegraphics[width=0.48\textwidth]{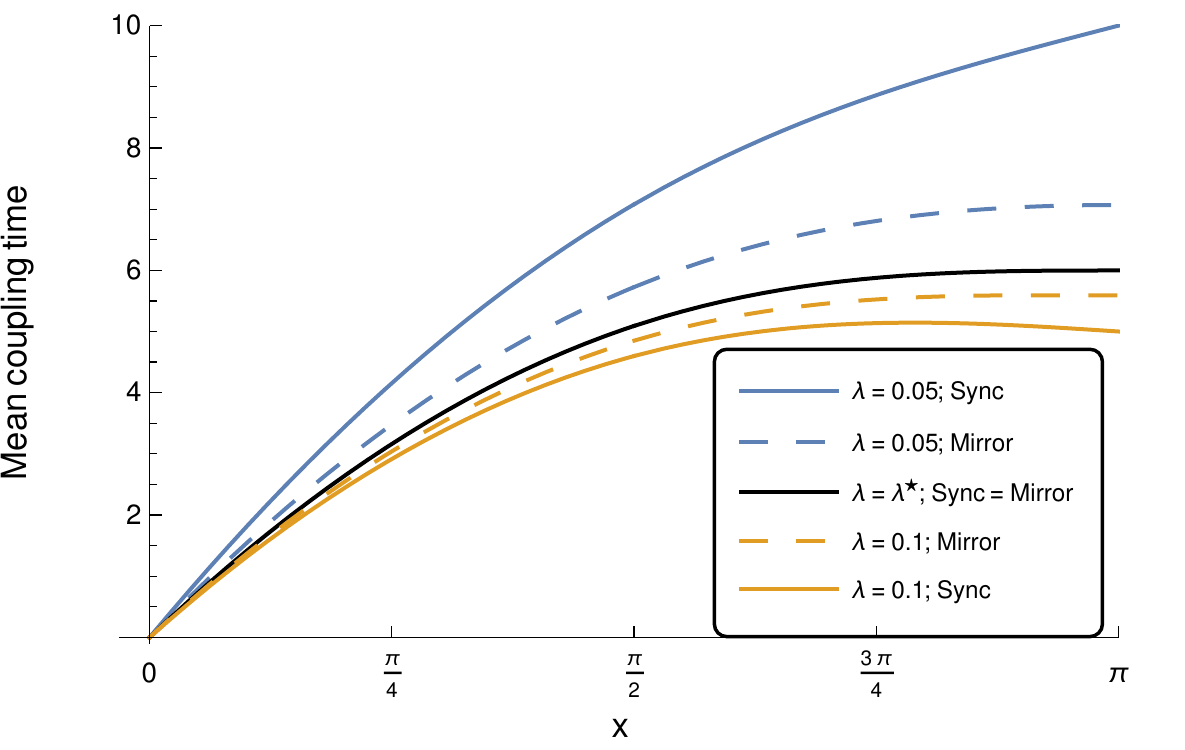}\hspace{5mm}
		\includegraphics[width=0.48\textwidth]{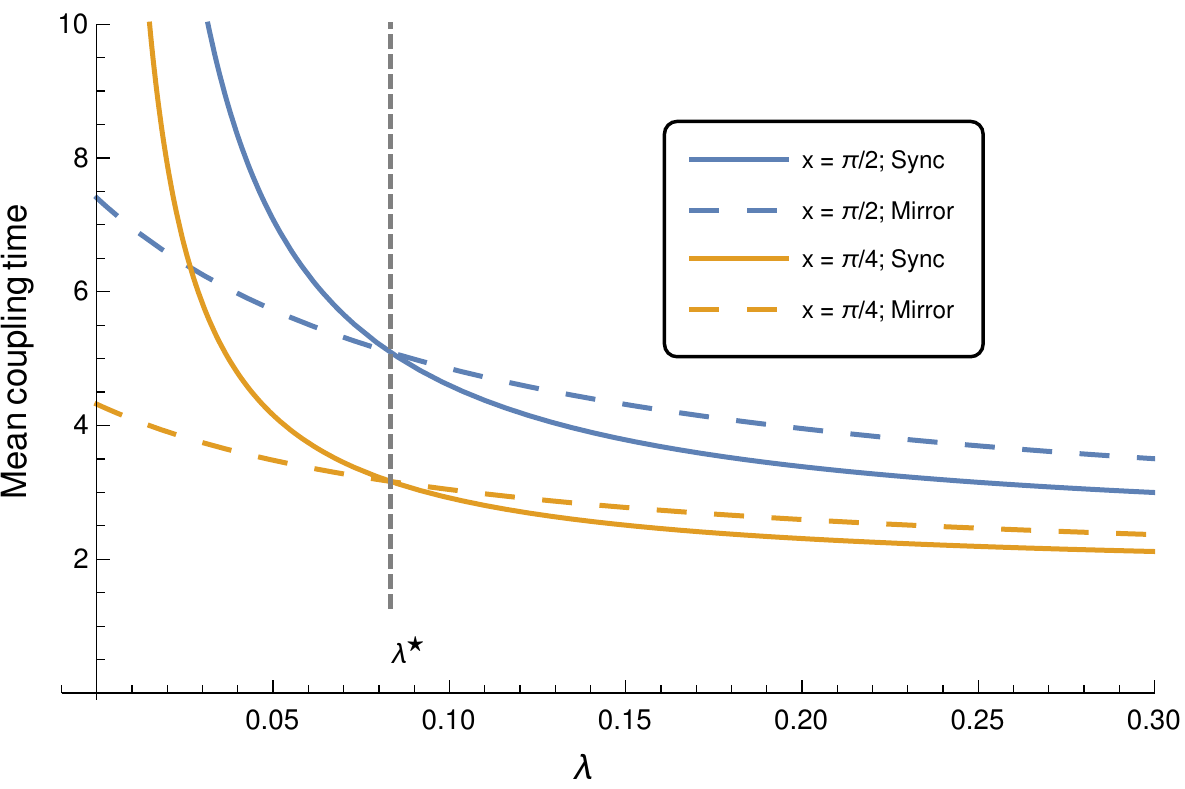}
	\caption{Expected coupling time under the Synchronous and Mirror couplings. The left-hand plot shows the mean coupling time as a function of $D_0=x$, for three representative values of $\lambda$. The right-hand plot shows the mean coupling time as a function of $\lambda$, for two particular values of $D_0$.}
\label{fig:expectation-comparison}
\end{figure}

Now let $\Tmin$ denote the coupling time of the mean-optimal co-adapted coupling (i.e. Mirror for $\lambda<\lambda^\star$, Synchronous otherwise). As expected, when $\lambda=0$ we have $\Ex[x]{\Tmin} = x(2\pi-x)$, and when $\lambda\to\infty$ we obtain $\Ex[x]{\Tmin} = x(\pi-x)$
(i.e. the mean time for a Brownian motion started at $x$ to hit $\{0,2\pi\}$ or $\{0,\pi\}$, respectively; recall the comment immediately following Theorem~\ref{thm:main_thm_mean}).
Figure~\ref{fig:expectation_RoW} plots $\Ex[x]{\Tmin}=\Ex[x]{T\mirror}\wedge\Ex[x]{T\sync}$ as a function of the initial distance $x\in[0,\pi]$, nicely showing the monotonic dependence upon the jump rate $\lambda$.
\begin{figure}[ht]
	\centering
	\includegraphics[height=5.5cm]{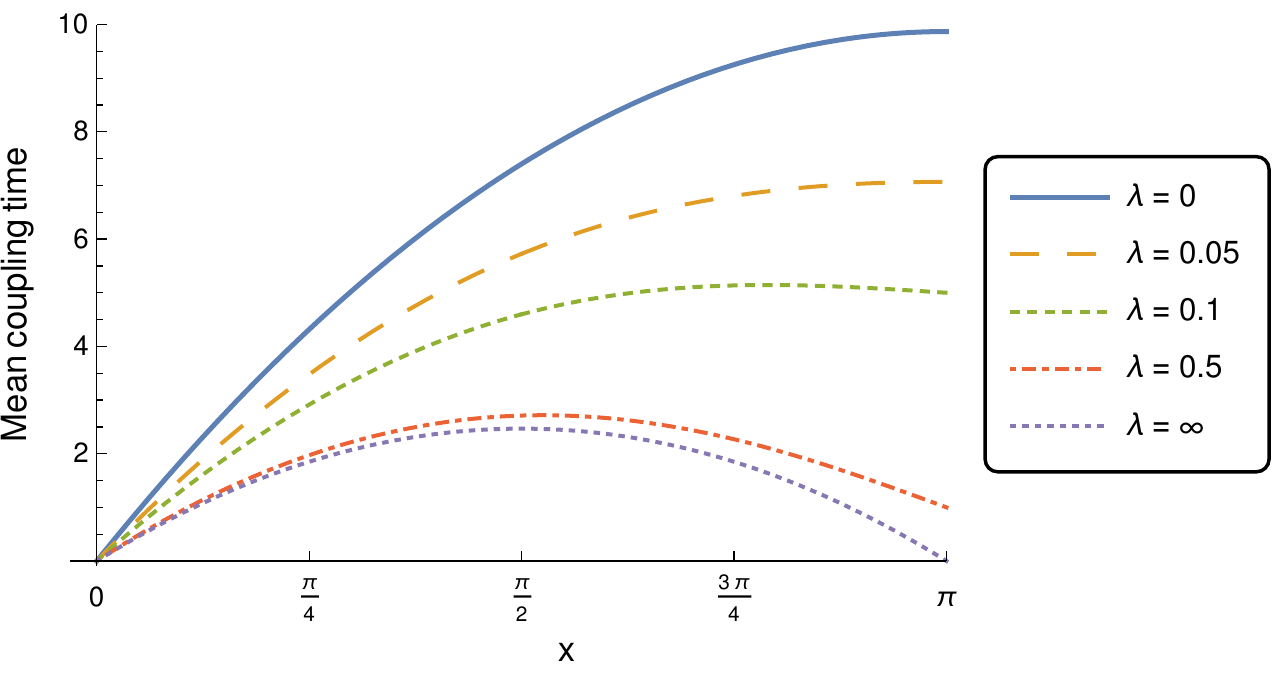}
	\caption{Expected coupling time from $D_0=x$ when using the mean-optimal strategy.}
	\label{fig:expectation_RoW}
\end{figure}

Finally, let us compare $\Tmin$ to a \emph{maximal} coupling. We know that the optimal co-adapted coupling is maximal when $\lambda=0$; given the construction of a maximal coupling in Section~\ref{sec:maximal} it should come as no surprise that this is not the case for any $\lambda>0$. From equation \eqref{eq:maxCouplingTail} we may easily calculate the Laplace transform and subsequently the mean of the maximal coupling time $T^\star$ when starting from opposite sides of the circle ($x=\pi$):
	\begin{equation}\label{eq:mean_maximal_pi}
		\Ex[\pi]{e^{-\gamma T^\star}} = \frac{1}{2\lambda+\gamma}\bra{2\lambda + \frac{\gamma}{\cosh(\pi\sqrt{2(2\lambda+\gamma)})}}  \,; \qquad \Ex[\pi]{T^\star} = \frac{1-\sech(2\pi\sqrt{\lambda})}{2\lambda} \,.
	\end{equation}
	On the other hand, Lemma~\ref{lem:expectation_RoW} tells us that $\Tmin$ satisfies
	\begin{equation}\label{eq:mean_RoW_pi}
		\Ex[\pi]{\Tmin} = \min\set{C\mirror(\lambda),C\sync(\lambda)}\frac{\sinh(\pi\sqrt{\lambda})}{\sqrt{\lambda}}\,.
	\end{equation}
Thus
\[ 
\frac{\Ex[\pi]{T^\star}}{\Ex[\pi]{\Tmin}} =
	\begin{cases}
		(\pi\sqrt\lambda)^{-1}(2\cosh(\pi\sqrt{\lambda})-1)\sinh(\pi\sqrt{\lambda})\sech(2\pi\sqrt{\lambda})
		& \quad 0\le\lambda< \lambda^\star\\
		1-\sech(2\pi\sqrt{\lambda}) & \quad\lambda^\star\le \lambda \,,
	\end{cases}
\]
and we note that for all values of $\lambda>0$ this ratio is strictly less than one. (For the case $\lambda<\lambda^\star$, this follows from the observation that
\[
	\frac{(2\cosh(z)-1)\sinh(z)\sech(2z)}{z} \,=\, \frac{\int_{z}^{2z}\cosh(y)\dd y}{z\cosh(2z)} \,,
\]
along with the convexity of $\cosh$.)
Graphs of the expressions in \eqref{eq:mean_maximal_pi} and \eqref{eq:mean_RoW_pi} are shown in Figure~\ref{fig:meanTimeComparison}: we see that the difference between the mean coupling times is largest when $\lambda = \lambda^\star$.
\begin{figure}[h]
	\centering
	\includegraphics[width=10cm]{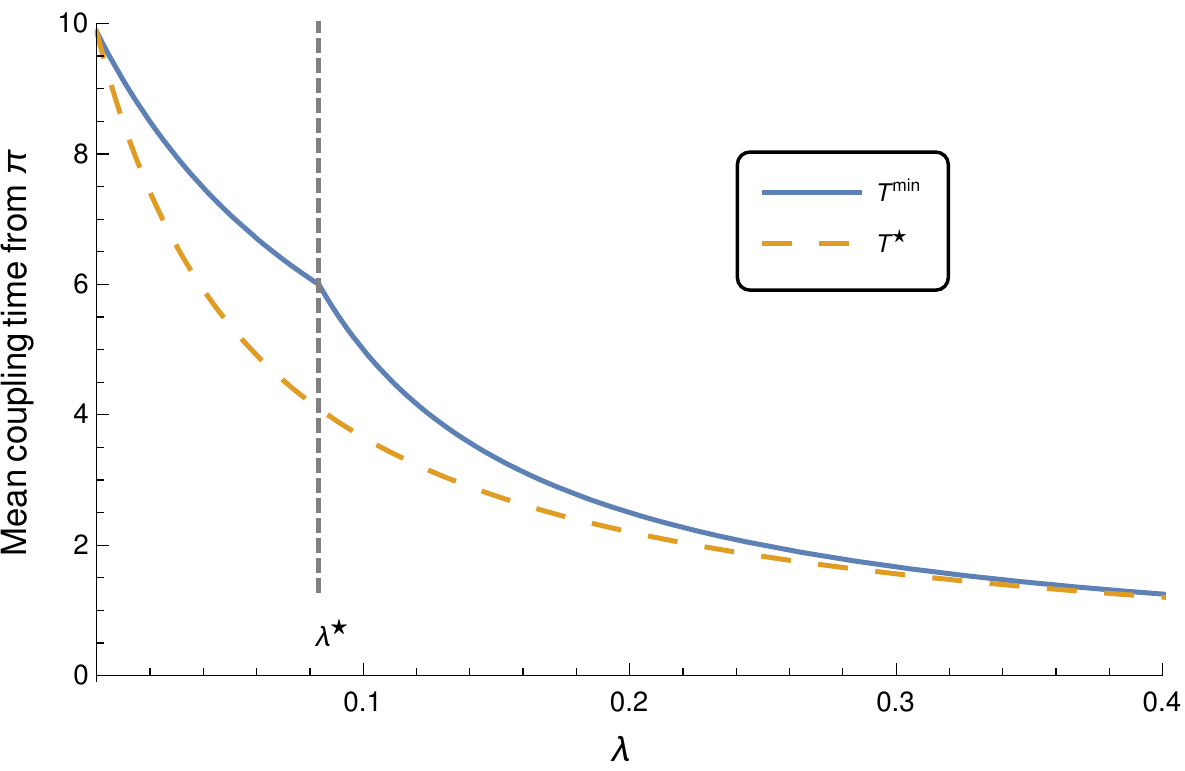}
	\centering
	\caption{Expectation of the mean-optimal co-adapted coupling time ($\Tmin$) and maximal coupling time ($T^\star$) when $D_0 = \pi$.}
	\label{fig:meanTimeComparison}
\end{figure}



\bibliographystyle{chicago} 
\bibliography{jbm}

\begin{thebibliography}{}

\bibitem[\protect\citeauthoryear{Aldous}{Aldous}{1983}]{aldousRandomWalksFinite1983}
Aldous, D. (1983).
\newblock Random walks on finite groups and rapidly mixing markov chains.
\newblock In J.~Az{\'e}ma and M.~Yor (Eds.), {\em S\'eminaire de
  {{Probabilit\'es XVII}} 1981/82}, Volume 986, pp.\  243--297. {Berlin,
  Heidelberg}: {Springer Berlin Heidelberg}.

\bibitem[\protect\citeauthoryear{Banerjee and Kendall}{Banerjee and
  Kendall}{2017}]{banerjeeRigidityMarkovianMaximal2017}
Banerjee, S. and W.~S. Kendall (2017, June).
\newblock Rigidity for {{Markovian}} maximal couplings of elliptic diffusions.
\newblock {\em Probability Theory and Related Fields\/}~{\em 168\/}(1),
  55--112.

\bibitem[\protect\citeauthoryear{Connor}{Connor}{2013}]{Connor-2013}
Connor, S.~B. (2013).
\newblock Optimal coadapted coupling for a random walk on the hyper-complete
  graph.
\newblock {\em Journal of Applied Probability\/}~{\em 50\/}(4), 1117--1130.

\bibitem[\protect\citeauthoryear{Connor and Jacka}{Connor and
  Jacka}{2008}]{Connor2008}
Connor, S.~B. and S.~D. Jacka (2008).
\newblock Optimal co-adapted coupling for the symmetric random walk on the
  hypercube.
\newblock {\em Journal of Applied Probability\/}~{\em 45\/}(3), 703--713.

\bibitem[\protect\citeauthoryear{Connor and Merli}{Connor and
  Merli}{2022}]{Connor.Merli-2022}
Connor, S.~B. and R.~Merli (2022).
\newblock Optimal coupling of jumpy {Brownian motion} on the circle.
\newblock {\em arxiv.org/abs/2203.14791\/}.

\bibitem[\protect\citeauthoryear{Goldstein}{Goldstein}{1979}]{goldsteinMaximalCoupling1979}
Goldstein, S. (1979).
\newblock Maximal coupling.
\newblock {\em Zeitschrift f\"ur Wahrscheinlichkeitstheorie und Verwandte
  Gebiete\/}~{\em 46\/}(2), 193--204.

\bibitem[\protect\citeauthoryear{Griffeath}{Griffeath}{1975}]{Griffeath-1975}
Griffeath, D. (1975).
\newblock A maximal coupling for {{Markov}} chains.
\newblock {\em Zeitschrift f\"ur Wahrscheinlichkeitstheorie und Verwandte
  Gebiete\/}~{\em 31\/}(2), 95--106.

\bibitem[\protect\citeauthoryear{Hsu and Sturm}{Hsu and Sturm}{2013}]{Hsu2013a}
Hsu, E.~P. and K.~T. Sturm (2013).
\newblock Maximal {{Coupling}} of {{Euclidean Brownian Motions}}.
\newblock {\em Communications in Mathematics and Statistics\/}~{\em 1\/}(1),
  93--104.

\bibitem[\protect\citeauthoryear{Jacka, Mijatović, and Širaj}{Jacka
  et~al.}{2014}]{jackaMirrorSynchronousCouplings2014}
Jacka, S.~D., A.~Mijatović, and D.~Širaj (2014).
\newblock Mirror and synchronous couplings of geometric {{Brownian}} motions.
\newblock {\em Stochastic Processes and their Applications\/}~{\em 124\/}(2),
  1055--1069.

\bibitem[\protect\citeauthoryear{Kendall}{Kendall}{2009}]{kendall2009brownian}
Kendall, W.~S. (2009).
\newblock Brownian couplings, convexity, and shy-ness.
\newblock {\em Electronic Communications in Probability\/}~{\em 14}, 66--80.

\bibitem[\protect\citeauthoryear{Kendall}{Kendall}{2015}]{kendall2015coupling}
Kendall, W.~S. (2015).
\newblock Coupling, local times, immersions.
\newblock {\em Bernoulli\/}~{\em 21\/}(2), 1014--1046.

\bibitem[\protect\citeauthoryear{Kuwada}{Kuwada}{2007}]{Kuwada2007}
Kuwada, K. (2007).
\newblock On uniqueness of maximal coupling for diffusion processes with a
  reflection.
\newblock {\em Journal of Theoretical Probability\/}~{\em 20\/}(4), 935--957.

\bibitem[\protect\citeauthoryear{Lindvall and Rogers}{Lindvall and
  Rogers}{1986}]{LindvallRogers-1986}
Lindvall, T. and L.~C.~G. Rogers (1986).
\newblock Coupling of multidimensional diffusions by reflection.
\newblock {\em The Annals of Probability\/}~{\em 14\/}(3), 860--872.

\bibitem[\protect\citeauthoryear{Merli}{Merli}{2021}]{Merli2021}
Merli, R. (2021).
\newblock {\em Probabilistic coupling: mixing times and optimality}.
\newblock Ph.\ D. thesis, Department of Mathematics, University of York.
  https://etheses.whiterose.ac.uk/30006.

\bibitem[\protect\citeauthoryear{{\O}ksendal and Sulem}{{\O}ksendal and
  Sulem}{2007}]{oksendal2007applied}
{\O}ksendal, B.~K. and A.~Sulem (2007).
\newblock {\em Applied stochastic control of jump diffusions}, Volume 498.
\newblock Springer.

\bibitem[\protect\citeauthoryear{Pitman}{Pitman}{1976}]{Pitman1976}
Pitman, J.~W. (1976).
\newblock On coupling of {{Markov}} chains.
\newblock {\em Zeitschrift f\"ur Wahrscheinlichkeitstheorie und Verwandte
  Gebiete\/}~{\em 35\/}(4), 315--322.

\bibitem[\protect\citeauthoryear{Rogers and Williams}{Rogers and
  Williams}{2000}]{rogers2000diffusions}
Rogers, L. C.~G. and D.~Williams (2000).
\newblock {\em Diffusions, Markov Processes and Martingales. Volume 2: It{\^o}
  Calculus}, Volume~2.
\newblock Cambridge University Press.

\bibitem[\protect\citeauthoryear{Sverchkov and Smirnov}{Sverchkov and
  Smirnov}{1990}]{sverchkovMaximalCouplingDvalued1990}
Sverchkov, M.~Y. and S.~N. Smirnov (1990).
\newblock Maximal coupling of {{D-valued}} processes.
\newblock In {\em Soviet {{Math}}. {{Dokl}}}, Volume~41.

\end{thebibliography}

\bigskip
\noindent
Department of Mathematics, University of York, York, YO10 5DD, UK. \\
Email: \url{stephen.connor@york.ac.uk} 

\end{document}